\documentclass{amsart}

\usepackage{amssymb,amsmath,amsthm,color}

\usepackage[bookmarksopen=true,bookmarksnumbered=true, bookmarksopenlevel=2, colorlinks=true,linkcolor=mblue,
citecolor=mblue,urlcolor=mblue, pdfstartview=FitH]{hyperref}

\definecolor{mblue}{rgb}{0,0,.8}

\newcommand{\N}{\mathbb N}
\newcommand{\Z}{\mathbb Z}
\newcommand{\Q}{\mathbb Q}
\newcommand{\Qbar}{\overline{\Q}}
\newcommand{\F}{\mathbb F}

\newcommand{\p}{\mathfrak p}
\newcommand{\Pp}{\mathfrak P}
\newcommand{\q}{\mathfrak q}
\newcommand{\Qq}{\mathfrak Q}

\newcommand{\OO}{\mathcal O}

\newtheorem{thm}{Theorem}
\newtheorem{lem}{Lemma}

\newtheorem{prop}{Proposition}
\newtheorem{cor}{Corollary}

\DeclareMathOperator{\GL}{GL} \DeclareMathOperator{\SL}{SL}  
\DeclareMathOperator{\Gal}{Gal}   
  \DeclareMathOperator{\Tr}{tr} \DeclareMathOperator{\ord}{ord}

\def\dash---{\thinspace---\hskip.16667em\relax}

\begin{document}

\title[On congruences mod ${\mathfrak p}^m$ between eigenforms.]{On congruences mod ${\mathfrak p}^m$ between eigenforms and their attached Galois representations.}
\author{Imin Chen, Ian Kiming and Jonas B. Rasmussen}
\address[Imin Chen]{Department of Mathematics, Simon Fraser University, 8888 University Drive, Burnaby, B.C., V5A 1S6, Canada}
\email{\href{mailto:ichen@math.sfu.ca}{ichen@math.sfu.ca}}
\address[Ian Kiming, Jonas B.\ Rasmussen]{Department of Mathematical Sciences, University of Copenhagen, Universitetsparken 5, 2100 Copenhagen \O , Denmark}
\email{\href{mailto:kiming@math.ku.dk}{kiming@math.ku.dk}}
\email{\href{mailto:jonas@math.ku.dk}{jonas@math.ku.dk}}



\begin{abstract} Given a prime $p$ and cusp forms $f_1$ and $f_2$ on some $\Gamma_1(N)$ that are eigenforms outside $Np$ and have coefficients in the ring of integers of some number field $K$, we consider the problem of deciding whether $f_1$ and $f_2$ have the same eigenvalues mod $\p^m$ (where $\p$ is a fixed prime of $K$ over $p$) for Hecke operators $T_{\ell}$ at all primes $\ell\nmid Np$.

When the weights of the forms are equal the problem is easily solved via an easy generalization of a theorem of Sturm. Thus, the main challenge in the analysis is the case where the forms have different weights. Here, we prove a number of necessary and sufficient conditions for the existence of congruences mod $\p^m$ in the above sense.

The prime motivation for this study is the connection to modular mod $\p^m$ Galois representations, and we also explain this connection.
\end{abstract}

\maketitle
\section{Introduction}\label{intro} Let $N\in\N$ and let $p$ be a fixed prime number.
\smallskip

Suppose that we are given cusp forms $f_1 = \sum a_n(f_1)q^n$ and $f_2 = \sum a_n(f_2)q^n$ (where $q := e^{2\pi iz}$) on $\Gamma_1(N)$ of weights $k_1$ and $k_2$, respectively, and with coefficients in $\OO_K$ where $K$ is some number field. We will assume in all that follows that $f_1$ and $f_2$ are normalized, i.e., that $a_1(f_1)=a_1(f_2)=1$.
\smallskip

We say that $f_1$ and $f_2$ are {\it eigenforms outside $Np$} if they are (normalized) eigenforms for all Hecke operators $T_{\ell}$ for primes $\ell$ with $\ell \nmid Np$. The corresponding eigenvalues for such $T_{\ell}$ acting on $f_i$ are then exactly the coefficients $a_{\ell}(f_i)$.
\smallskip

Now fix a prime $\p$ of $K$ over $p$. If $f_i$ is an eigenform outside $Np$, and if $m\in\N$, there is attached to $f_i$ a `mod $\p^m$' Galois representation:
$$
\rho_{f_i,\p^m} : ~ G_{\Q} := \Gal(\Qbar/\Q) \rightarrow \GL_2(\OO_K/\p^m)
$$
obtained by making the $p$-adic representation attached to $f_i$ integral with coefficients in $\OO_K$ and then reducing modulo $\p^m$. The representation $\rho_{f_i,\p^m}$ is unramified outside $Np$ and we have:
$$
\Tr \rho_{f_i,\p^m}(\mathrm{Frob}_{\ell}) = (a_{\ell}(f_i) \bmod{\p^m}) \leqno{(\ast)}
$$
for primes $\ell\nmid Np$.

By a theorem of Carayol, cf.\ Th\'{e}or\`{e}me 1 of \cite{Carayol}, combined with the Chebotarev density theorem, the representation $\rho_{f_i,\p^m}$ is determined up to isomorphism by the property $(\ast)$ for primes $\ell\nmid Np$ if we additionally suppose that the mod $\p$ representation $\rho_{f_i,\p}$ is absolutely irreducible.
\smallskip

Motivated by a study of the arithmetic properties of modular mod $\p^m$ Galois representations \cite{icik}, we found it natural to prepare the ground for numerical experimentation with these representations. As is obvious from the above, the key to this is to obtain a computationally decidable criterion for when we have $a_{\ell}(f_1) \equiv a_{\ell}(f_2) \pod{\p^m}$ for all primes $\ell\nmid Np$, if $f_1$ and $f_2$ as above are given cusp forms that are eigenforms outside $Np$.
\smallskip

Now, for the case $m=1$, and if the weights $k_1$ and $k_2$ are equal, there is a well-known theorem of Sturm that gives a necessary and sufficient condition for the forms to be congruent mod $\p$ in the sense that all their Fourier coefficients are congruent mod $\p$. It turns out to be very easy to generalize Sturm's theorem to the cases $m>1$ provided that we still have $k_1=k_2$. Then, still under the assumption that the weights are equal, a simple twisting argument allows us to discuss the case of eigenforms outside $Np$.
\smallskip

For various reasons we are interested in also considering cases where the weights are distinct and this turns out to present a genuinely new challenge.
\smallskip

We study two distinct approaches to this challenge. Under favorable circumstances these approaches both result in computable necessary and sufficient conditions for the forms to be `congruent mod $\p^m$ outside $Np$' in the above sense.

The first approach is to generalize a theorem of Serre-Katz on $p$-adic modular forms, cf.\ Cor.\ 4.4.2 of \cite{Katz} which -- under certain restrictions on the levels of the forms -- gives a necessary congruence between the weights for the forms to be congruent mod $\p^m$. In the Serre-Katz theorem one needs to assume that the prime $\p$ of the field $K$ of coefficients is unramified relative to $p$ in $\Q$. We are able to generalize this theorem to cases where $\p$ is ramified over $p$.

Under certain technical restrictions, in particular that the ramification index relative to $p$ of the Galois closure of the field $K$ of coefficients is not divisible by $p$, and that $p$ is odd, our Theorem \ref{theorem1} results in the desired computable necessary and sufficient conditions. See Corollary \ref{cor1} below.

The second approach is via a study of the determinants of the attached mod $\p^m$ representations. Again under certain technical restrictions, here notably a restriction on the nebentypus characters of the forms, our Theorem \ref{theorem2} leads to the desired computable necessary and sufficient conditions. Cf.\ Corollary \ref{cor2} below.

It is remarkable that these two rather distinct approaches result -- under the technical restrictions alluded to above -- in necessary and sufficient conditions that are close to being equivalent.
\smallskip

We illustrate the results by a few numerical examples.
\smallskip

Finally, let us mention that Kohnen has considered similar questions, but only in the mod $\p$ case, see \cite{kohnen}, and our results can also be regarded as a generalization of Kohnen's results to the mod $\p^m$ setting.

\subsection{Notation} To formulate our results, let us introduce the following notation:
\smallskip

Define
$$
N' := \left\{ \begin{array}{ll} N\cdot \prod_{q\mid N} q ~,& \mbox{if $p\mid N$}\\ & \\ N\cdot p^2\cdot \prod_{q\mid N} q ~,& \mbox{if $p\nmid N$} \end{array} \right.
$$
where the products are over prime divisors $q$ of $N$. Put:
$$
\mu := [\SL_2(\Z):\Gamma_1(N)] ~,\quad \mu' := [\SL_2(\Z):\Gamma_1(N')] ~,
$$
and fix the following notation:
$$
\begin{array}{lll}
m & : & \mbox{a natural number},\\
k & := & \max\{k_1,k_2\},\\
\p & : & \mbox{a fixed prime of $K$ over $p$},\\
e & := & e(\p/p), \mbox{ the ramification index of $\p$ over $p$},\\
L & : & \mbox{Galois closure of $K/\Q$},\\
e(L,p) & : & \mbox{the ramification index of $L$ relative to $p$ in $\Q$},\\
r & : & \mbox{largest power of $p$ dividing the ramification index $e(L,p)$},\\
\ell & : & \mbox{a (not fixed) prime number}.
\end{array}
$$

For a natural number $a$ and a modular form $h=\sum c_n q^n$ on some $\Gamma_1(M)$ and coefficients $c_n$ in $\OO_K$ we define:
$$
\ord_{\p^a}h = \inf\big\{n \mid ~ \p^a \nmid (c_n) \big\},
$$
with the convention that $\ord_{\p^a}h = \infty$ if $\p^a \mid (c_n)$ for all $n$.

We say that $f_1$ and $f_2$ are congruent modulo $\p^a$ if $\ord_{\p^a}(f_1-f_2) = \infty$, and we denote this by $f_1 \equiv f_2 \pod {\p^a}$.

\subsection{Results} The following proposition is the first, basic observation, and is an easy generalization of a well-known theorem of Sturm, cf.\ \cite{Sturm}.

\begin{prop} \label{proposition1} Suppose that $N$ is arbitrary, but that $f_1$ and $f_2$ are forms on $\Gamma_1(N)$ of the same weight $k=k_1=k_2$ and coefficients in $\OO_K$.
\smallskip

Then $\ord_{\p^m}(f_1-f_2) > k\mu/12$ implies $f_1 \equiv f_2 \pod {\p^m}$.
\end{prop}

Part {\it (i)} of the following theorem is a slight generalization of theorems of Serre and Katz, cf.\ \cite{Serre}, Th\'{e}or\`{e}me 1, \cite{Katz}, \ Corollary 4.4.2.

\begin{thm}\label{theorem1} Assume $N\ge 3$, and let $f_1$ and $f_2$ be normalized cusp forms on $\Gamma_1(N)$ of weights $k_1$ and $k_2$, respectively, and with coefficients in $\OO_K$.
\smallskip

\noindent (i) Assume additionally that $p\nmid N$ and that $f_1$ and $f_2$ are forms on $\Gamma_1(N) \cap \Gamma_0(p)$.
\smallskip

Then if $f_1 \equiv f_2 \pod {\p^m}$ we have $k_1 \equiv k_2 \pod {p^s(p-1)}$ with the non-negative integer $s$ defined as follows:
$$
s := \left\{ \begin{array}{ll} \max\{ 0, \lceil \frac{m}{e} \rceil - 1 - r\} ~,& \mbox{if $p\ge 3$} \\ \max\{ 0, \alpha(\lceil \frac{m}{e} \rceil  - r ) \} ~,& \mbox{if $p=2$} \end{array} \right.
$$
with $\alpha(u)$ defined for $u\in\Z$ as follows:
$$
\alpha(u) := \left\{ \begin{array}{ll} u-1 ~,& \mbox{if $u\le 2$} \\ u-2 ~,& \mbox{if $u\ge 3$} ~.\end{array} \right.
$$
\medskip

\noindent (ii) Let $N$ be arbitrary, but assume $3\mid N$ if $p=2$, and $2\mid N$ if $p=3$.
\smallskip

Suppose that $k_1 \equiv k_2 \pod {p^s(p-1)}$.
\smallskip

Then, if $a_{\ell}(f_1) \equiv a_{\ell}(f_2) \pod {\p^m}$ for all primes $\ell \leq k\mu' /12$ with $\ell \nmid Np$, we have
$$
a_{\ell}(f_1) \equiv a_{\ell}(f_2) \pod {\p^{\min \{ e\cdot (s+1) , m \} }}
$$
for all primes $\ell \nmid Np$.
\smallskip

In particular, if either $m\le e$ or ($p>2$ and $r=0$) or ($p=2$, $r=0$, and $m\le 2e$), we have
$$
a_{\ell}(f_1) \equiv a_{\ell}(f_2) \pod {\p^m}
$$
for all primes $\ell \nmid Np$.
\end{thm}

The following corollary is an immediate consequence of Theorem \ref{theorem1}.

\begin{cor}\label{cor1} Retain the setup and notation of Theorem \ref{theorem1}, and assume that $p$ is odd, $r=0$, that $N$ is prime to $p$, that $3\mid N$ if $p=2$, and $2\mid N$ if $p=3$, and that $f_1$ and $f_2$ are forms on $\Gamma_1(N) \cap \Gamma_0(p)$.
\smallskip

Then we have $a_{\ell}(f_1) \equiv a_{\ell}(f_2) \pod {\p^m}$ for all primes $\ell \nmid Np$ if and only if this congruence holds for all primes $\ell \leq k\mu' /12$ with $\ell \nmid Np$ and we have the congruence
$$
k_1 \equiv k_2 \pod {p^s(p-1)}
$$
between the weights.
\end{cor}

\begin{thm}\label{theorem2} Suppose that $N$ is arbitrary, but assume that $p$ is odd and that $f_1$ and $f_2$ are normalized cusp forms on $\Gamma_1(N)$ of weights $k_1$ and $k_2$ and with nebentypus characters $\psi_1$ and $\psi_2$, respectively.

Suppose that $f_1$ and $f_2$ are eigenforms outside $Np$ and have coefficients in $\OO_K$, and that the mod $\p$ Galois representation attached to $f_1$ is absolutely irreducible.
\smallskip

View the nebentypus characters $\psi_i$ as finite order characters on $G_{\Q}$, and let the order of the character
$$
\left( \psi_2 \psi_1^{-1} \bmod{\p^m} \right) _{\mid I_p}
$$
where $I_p$ is an inertia group at $p$, be $p^{\delta}\cdot d$ with $d$ a divisor of $p-1$.
\medskip

\noindent (i) If we have $a_{\ell}(f_1) \equiv a_{\ell}(f_2) \pod {\p^m}$ for all primes with $\ell \nmid Np$ then $\delta \le \lceil \frac{m}{e} \rceil - 1$ and we have:
$$
k_1 \equiv k_2 \pmod{p^{\lceil \frac{m}{e} \rceil - 1 - \delta} \cdot (p-1)/d}
$$
so that in particular, $k_1 \equiv k_2 \pmod{p^{\lceil \frac{m}{e} \rceil - 1} \cdot (p-1)/d}$ if $\delta = 0$.
\smallskip

\noindent (ii) Suppose that
$$
k_1 \equiv k_2 \pmod{p^{\lceil \frac{m}{e} \rceil - 1} \cdot (p-1)/d} ~.
$$

Then, if $a_{\ell}(f_1) \equiv a_{\ell}(f_2) \pod {\p^m}$ for all primes $\ell \leq k\mu' /12$ with $\ell \nmid Np$ we have this congruence for all primes $\ell \nmid Np$.
\end{thm}

The following corollary follows immediately from Theorem \ref{theorem2}.

\begin{cor}\label{cor2} Retain the setup and notation of Theorem \ref{theorem2}, and assume that $\delta = 0$.
\smallskip

Then we have $a_{\ell}(f_1) \equiv a_{\ell}(f_2) \pod {\p^m}$ for all primes $\ell \nmid Np$ if and only if this congruence holds for all primes $\ell \leq k\mu' /12$ with $\ell \nmid Np$ and we have the congruence
$$
k_1 \equiv k_2 \pod {p^{\lceil \frac{m}{e} \rceil -1} \cdot (p-1)/d}
$$
between the weights.
\end{cor}

\noindent {\bf Remark:} Obtaining results like those in the corollaries, but in more general situations, for instance with $r$ and $\delta$ not necessarily $0$, are obvious problems for future work. We suspect such questions will be more involved.

\section{Proofs} Let us first prove Proposition \ref{proposition1} that turns out to be an easy generalization of a theorem by Sturm, cf.\ \cite{Sturm}.

\begin{proof}[Proof of Proposition \ref{proposition1}:] We prove this by induction on $m$. It will be convenient to prove a slightly more general statement, namely that the proposition holds for forms with coefficients in $(\OO_K)_{\p}$, the localization of $\OO_K$ w.r.t.\ $\p$: If $h$ is such a form we can define $\ord_{\p^m}(h)$ in the same manner as above, and the claim is then that $\ord_{\p^m}(h) > k\mu/12$ implies $\ord_{\p^m}(h) = \infty$.
\smallskip

This statement for $m=1$ follows immediately from a theorem of Sturm, cf.\ Theorem 1 of \cite{Sturm}: If $h$ is a form on $\Gamma_1(N)$ of weight $k$ and coefficients in $(\OO_K)_{\p}$ then there is a number $\alpha\in \OO_K \backslash \p$ such that $\alpha\cdot h$ has coefficients in $\OO_K$; this follows from the `bounded denominators' property for modular forms. Then, if $\ord_{\p^m}(h) > k\mu/12$ we have $\ord_{\p^m}(\alpha\cdot h) > k\mu/12$ and by Sturm this implies $\ord_{\p^m}(\alpha\cdot h) = \infty$ and so also $\ord_{\p^m}(h) = \infty$.
\smallskip

Assume that $m > 1$, and that the proposition in the above slightly more general form is true for powers $\p^a$ of $\p$ with $a<m$. Consider then forms $f_1$ and $f_2$ on $\Gamma_1(N)$ of weight $k$ with coefficients in $(\OO_K)_{\p}$ such that $\ord_{\p^m}(f_1-f_2) > k\mu/12$. Let $\varphi = f_1-f_2$. By assumption we have $\ord_{\p^m}\varphi > k\mu/12$, and therefore also $\ord_{\p^{m-1}}\varphi > k\mu/12$, and hence the induction hypothesis gives $\ord_{\p^{m-1}}\varphi = \infty$. Choose a uniformizer $\pi$ for $\p$, i.e., an element $\pi \in \p \backslash \p^2$.

We see that the form
$$
\psi := \frac{1}{\pi^{m-1}} \cdot \varphi
$$
is a form on $\Gamma_1(N)$ of weight $k$ with coefficients in $(\OO_K)_{\p}$.

Since $\ord_{\p^m}\varphi > k\mu/12$, we must have $\ord_{\p}\psi > k\mu/12$, so that $\ord_{\p}\psi = \infty$ by the induction hypothesis for $m=1$. From this we conclude that $\ord_{\p^m}\varphi = \infty$, as desired.
\end{proof}

In subsequent arguments we occasionally need the following simple and probably well-known lemma.

\begin{lem}\label{lemma}
Let $F'/F$ be a finite extension of number fields. Let $\q$ be a prime ideal of $F$ and let $\Qq$ be a prime ideal of $F'$ over $\q$ of ramification index $\epsilon$. Let $b$ be a positive integer.
\smallskip

Then
$$
\Qq^b \cap F = \q^{\lceil\frac{b}{\epsilon}\rceil}.
$$
\end{lem}

\begin{proof}
There is a non-negative integer $a$ such that $a\epsilon < b \leq (a+1)\epsilon$, and then we
have
$$
\Qq^{(a+1)\epsilon} \subseteq \Qq^b \subseteq \Qq^{a\epsilon} ~.
$$
From this we get that
$$
\q^{a+1} = \Qq^{(a+1)\epsilon} \cap F \subseteq \Qq^b \cap F \subseteq \Qq^{a\epsilon} \cap F = \q^a ~,
$$
and so $\Qq^b \cap F$ is either $\q^a$ or $\q^{a+1}$.
\smallskip

Assume that $\Qq^b \cap F = \q^a$. Then $\q^a \subseteq \Qq^b$, i.e., $\Qq^{a\epsilon} \subseteq \Qq^b$, and so $a\epsilon \geq b$, a contradiction. We conclude that $\Qq^b \cap F = \q^{a+1}$, and since $a+1 = \lceil\frac{b}{\epsilon}\rceil$ by the definition of $a$, we are done.
\end{proof}

Part {\it (i)} of Theorem \ref{theorem1} can be seen as a generalization of a theorem of Serre and Katz, cf.\ Cor.\ 4.4.2 of \cite{Katz}, and Katz' theorem is also the main point of the proof.

\begin{proof}[Proof of part (i) of Theorem \ref{theorem1}:]
Recall that $L$ denotes the Galois closure of $K$. Let us fix a prime $\Pp$ over $\p$ in the Galois closure $L$ of $K$. Thus, the ramification index $e(L,p)$ is the ramification index of $e(\Pp/p)$ of $\Pp$ relative to $p$ in $\Q$. Recall that we denote the ramification index $e(\p/p)$ by $e$.

Let $L_0$ be the subfield of $L$ corresponding to the inertia group $I(\Pp/p)$. Let $\p_0$ be the prime of $L_0$ under $\Pp$.
\smallskip

We now let $I(\Pp/p)$ act on the $f_i$ by acting on their Fourier coefficients. Since $f_1 \equiv f_2 \pod {\p^m}$ we have $\sigma(f_1) \equiv \sigma(f_2) \pod {\Pp^{m\cdot e(\Pp/\p)}}$ for all $\sigma \in I(\Pp/p)$. Letting
$$
F_1 = \sum_{\sigma}\sigma(f_1) \quad \mbox{and} \quad F_2 = \sum_{\sigma}\sigma(f_2)
$$
with the sums taken over all $\sigma \in I(\Pp/p)$, we therefore obtain
$$
F_1 \equiv F_2 \pod {\Pp^{m\cdot e(\Pp/\p)}} ~.
$$

Now, since $F_1$ and $F_2$ are invariant under the action of $I(\Pp/p)$ they actually have coefficients in $L_0$, and we therefore have
$$
F_1 \equiv F_2 \pod {\p_0^{\lceil\frac{m}{e}\rceil}}
$$
since $\Pp^b \cap L_0 = \p_0^{\lceil\frac{b}{e(L,p)}\rceil}$ for non-negative integers $b$, cf.\ Lemma \ref{lemma}, and because
$$
e(L,p) = e(\p/p)e(\Pp/\p) = e\cdot e(\Pp/\p) ~.
$$

Now, the extension $(L_0)_{\p_0}/\Q_p$ of local fields is unramified, and so $(L_0)_{\p_0}$ is the field of fractions of the ring $W = W(\F_{p^f})$ of Witt vectors over $\F_{p^f}$ for some $f$. Since the $F_i$ have integral coefficients in $L_0$, we can view them as having coefficients in $W$.

Now let $a$ be the largest non-negative integer such that all Fourier coefficients of $F_1$ and $F_2$ are divisible by $p^a$. Then the forms $p^{-a}F_1$ and $p^{-a}F_2$ are cusp forms on $\Gamma_1(N) \cap \Gamma_0(p)$ of weights $k_1$ and $k_2$, respectively, and with coefficients in $W$. At least one of these forms has a $q$-expansion that does not reduce to $0$ identically modulo $p$. Their $q$-expansions are congruent modulo
$$
\p_0^{\max \{ 0, \lceil\frac{m}{e}\rceil -a \} }
$$
and hence also modulo
$$
\p_0^{\max \{ 0, \lceil\frac{m}{e}\rceil - r \} }
$$
since certainly $a\le r$ because the coefficients of $q$ for both forms $F_i$ equals $\# I(\Pp/p)$ which is just $e(L,p)$.

By a theorem of Katz, see Cor.\ 4.4.2 as well as Theorem 3.2 of \cite{Katz}, we then deduce that
$$
k_1 \equiv k_2 \pod{p^s(p-1)}
$$
where $s$ is given as in the theorem. Notice that we need our hypothesis $N\ge 3$ because of this reference to \cite{Katz}.
\end{proof}

To prepare for the proof of part {\it (ii)} of Theorem \ref{theorem1} we need the following lemma.
\smallskip

Let us say that a cusp form $h=\sum c_n q^n$ on $\Gamma_1(N)$ and coefficients in $\OO_K$ is an {\it eigenform mod $\p^m$ outside $Np$} if it is normalized and we have $T_{\ell}h \equiv \lambda_{\ell} h \pod{\p^m}$ for all primes $\ell\nmid Np$ with certain $\lambda_{\ell}\in \OO_K$. The same argument as in characteristic $0$ shows that in that case, the mod $\p^m$ eigenvalues $\lambda_{\ell}$ are congruent mod $\p^m$ to the Fourier coefficients $c_{\ell}$.

\begin{lem}\label{twist0} Let $N$ be arbitrary and let $f_1$ and $f_2$ be normalized forms of the same weight $k$ on $\Gamma_1(N)$ and with coefficients in $\OO_K$.

Suppose that $f_1$ and $f_2$ are eigenforms mod $\p^m$ outside $Np$ such that
$$
a_{\ell}(f_1) \equiv a_{\ell}(f_2) \pod {\p^m}
$$
for all primes $\ell \leq k\mu' /12$ with $\ell \nmid Np$.
\smallskip

Then $a_{\ell}(f_1) \equiv a_{\ell}(f_2) \pod {\p^m}$ for all primes $\ell \nmid Np$.
\end{lem}
\begin{proof} We first apply Lemma 4.6.5 of Miyake \cite{Miyake}: By that lemma we obtain from the $f_i$ forms $f_i'$ of weight $k$ on $\Gamma_1(N')$ by putting:
$$
f_i' := \sum_{\gcd(n,Np)=1} a_n(f_i) \cdot q^n ~.
$$

Here, $N'$ is as defined in the notation section. The forms $f_i'$ obviously still have coefficients in $\OO_K$.
\smallskip

Now, since the $f_i$ are eigenforms mod $\p^m$ outside $Np$, the forms $f_i'$ are also eigenforms mod $\p^m$ outside $Np$, with the same eigenvalues $(a_{\ell}(f_i) \bmod{\p^m})$.

On the other hand, all Fourier coefficient of the forms $f_i'$ at any index $n$ not prime to $Np$ vanishes. By our hypotheses we can thus conclude that
$$
\ord_{\p^m}(f_1'-f_2') > k\mu'/12
$$
and by Proposition \ref{proposition1} this implies $f_1'\equiv f_2' \pod{\p^m}$.
\smallskip

But then $a_{\ell}(f_1) \equiv a_{\ell}(f_2) \pod {\p^m}$ for all primes $\ell \nmid Np$.
\end{proof}

\begin{proof}[Proof of part (ii) of Theorem \ref{theorem1}] Assume without loss of generality that $k_2 \ge k_1$. We can then write:
$$
k_2 = k_1 + t \cdot p^s(p-1)
$$
where $t$ is a non-negative integer.

Now, we have an Eisenstein series $E$ of weight $p-1$ on $\Gamma_1(N)$ with coefficients in $\Z$ and such that $E\equiv 1 \pod{p}$: If $p\ge 5$ we can take $E:=E_{p-1}$ the standard Eisenstein series of weight $p-1$ on $\SL_2(\Z)$. If $p=2$ there is, cf.\ \cite{diamond-shurman} chap.\ 4.8 for instance, an Eisenstein series of weight $1$ on $\Gamma_1(3)$:
$$
E:= 1 - \frac{2}{B_{1,\psi}} \cdot \sum_{n=1}^{\infty} \left( \sum_{d\mid n} \psi(d) \right) \cdot q^n ~;
$$
here, $\psi$ is the primitive Dirichlet character of conductor $3$, and $B_{1,\psi}$ is the first Bernoulli number of $\psi$. One computes $B_{1,\psi} = -\frac{1}{3}$, so that in fact $E$ has coefficients in $\Z$ and reduces to $1$ modulo $2$. Also, $E$ is a modular form on $\Gamma_1(N)$ as we have assumed $3\mid N$ if $p=2$.

If $p=3$ we choose
$$
E:= 1 - 24 \cdot \sum_{n=1}^{\infty} \left( \sum_{d\mid n} d \right) \cdot q^n ~;
$$
this is a modular form of weight $2$ on $\Gamma_1(2)$ and hence also on $\Gamma_1(N)$ as we have $2\mid N$ if $p=3$. Again, cf.\ for instance \cite{diamond-shurman}, chap. 4.6.
\smallskip

With the above choice of $E$ we have in all cases that $E$ is a modular form of weight $p-1$ on $\Gamma_1(N)$ with coefficients in $\Z$ that reduces to $1$ modulo $p$. By induction on $j$ we see that $E^{p^j} \equiv 1 \pod{p^{j+1}}$ for all non-negative integers $j$, and hence also:
$$
E^{t\cdot p^s} \equiv 1 \pod{p^{s+1}}
$$
that we write as $E^{t\cdot p^s} \equiv 1 \pod{\p^{e\cdot (s+1)}}$. Consequently, the form
$$
\tilde{f} := E^{t\cdot p^s} \cdot f_1
$$
satisfies $\tilde{f}\equiv f_1 \pod{\p^{e\cdot (s+1)}}$. If we call $\tilde{a}_n$ the Fourier coefficients of $\tilde{f}$ we have then
$$
\tilde{a}_n \equiv a_n(f_1) \pod{\p^{e\cdot (s+1)}}
$$
and thus consequently:
$$
\tilde{a}_{\ell} \equiv a_{\ell}(f_2) \pod{\p^{\min \{ e\cdot (s+1) , m \} }}
$$
for all primes $\ell \leq k\mu' /12$ with $\ell \nmid Np$, because of our hypothesis on $f_1$ and $f_2$.
\smallskip

Now, $\tilde{f}$ and $f_2$ are both forms on $\Gamma_1(N)$ of weight $k=k_2$. As $f_1$ is an eigenform mod $\p^m$ outside $Np$, we have that $\tilde{f}$ and $f_2$ are both eigenforms mod $\p^{\min \{ e\cdot (s+1) , m \} }$ outside $Np$. Thus, Lemma \ref{twist0} implies that
$$
\tilde{a}_{\ell} \equiv a_{\ell}(f_2) \pod{\p^{\min \{ e\cdot (s+1) , m \} }}
$$
and hence also
$$
a_{\ell}(f_1) \equiv a_{\ell}(f_2) \pod{\p^{\min \{ e\cdot (s+1) , m \} }}
$$
for all primes $\ell \nmid Np$.
\smallskip

Using the definition of $s$ one checks that if either $m\le e$ or ($p>2$ and $r=0$) or ($p=2$, $r=0$, and $m\le 2e$) then we have $e\cdot (s+1) \ge m$. In each of those cases we thus have
$$
a_{\ell}(f_1) \equiv a_{\ell}(f_2) \pod{\p^m}
$$
for all primes $\ell \nmid Np$.
\end{proof}

\begin{proof}[Proof of Theorem \ref{theorem2}:] {\it Proof of part (i):} Consider the representations $\rho_{f_i,\p^m}$ attached to the forms $f_i$.
\smallskip

Since $a_{\ell}(f_1) \equiv a_{\ell}(f_2) \pod {\p^m}$ for all primes $\ell \nmid Np$ we can conclude by Chebotarev's density theorem that the representations $\rho_{f_1,\p^m}$ and $\rho_{f_2,\p^m}$ have the same traces. As $\rho_{f_1,\p}$ is assumed absolutely irreducible, Th\'{e}or\`{e}me 1 of Carayol \cite{Carayol} then implies that $\rho_{f_1,\p^m}$ and $\rho_{f_2,\p^m}$ are isomorphic. Hence, the determinants of these representations are also isomorphic. These determinants are:
$$
\det\rho_{f_i,\p^m} = \left( \psi_i \cdot \chi^{k_i-1} \bmod{\p^m} \right)
$$
where $\chi$ denotes the $p$-adic cyclotomic character $\chi :~ G_{\Q} \rightarrow \Z_p^{\times}$, and the nebentypus characters $\psi_i$ are now seen as finite order characters on $G_{\Q}$. Observe that the characters $\psi_i$ take values in $\OO_K$ so that it makes sense to reduce them mod $\p^m$. Also, reducing $\chi$ mod $\p^m$ is to be taken in the obvious sense.
\smallskip

We can now deduce that
$$
\left( \psi_2 \psi_1^{-1} \bmod{\p^m} \right) _{\mid I_p} = \left( \chi \bmod{\p^m} \right) ^{k_1-k_2} _{\mid I_p} ~.
$$

Now let us view via local class field theory the character $\left( \chi \bmod{\p^m} \right)_{\mid I_p}$ as a character on $\Z_p^{\times}$. As such it factors through $(\Z /\Z p^{\lceil \frac{m}{e} \rceil})^{\times}$ and has order
$$
p^{\lceil \frac{m}{e} \rceil -1} \cdot (p-1) ~;
$$
cf.\ Lemma \ref{lemma}. By definition, the character $\left( \psi_2 \psi_1^{-1} \bmod{\p^m} \right) _{\mid I_p}$ has order $p^{\delta} \cdot d$ with $d$ a divisor of $p-1$. Hence, first we see that $p^{\delta} \cdot d$ is a divisor of $p^{\lceil \frac{m}{e} \rceil -1} \cdot (p-1)$ which implies that $\delta \le \lceil \frac{m}{e} \rceil -1$. Secondly, we then conclude that $k_1-k_2$ is divisible by $p^{\lceil \frac{m}{e} \rceil - 1 - \delta} \cdot (p-1)/d$ as desired.
\medskip

\noindent {\it Proof of part (ii):} Observe first the following. If $p\nmid N$ then upon replacing $N$ by $Np$ and then calculating $\mu'$ we end up with the same number $\mu'$ as had we calculated it from $N$. And of course our forms are also forms on $\Gamma_1(Np)$.

This means that we may well assume that $N$ is divisible by $p$, -- our hypotheses remain unchanged when $N$ is replaced by $Np$ if $N$ is not divisible by $p$.

In particular, we may assume that the group $\Gamma_1(N)$ is contained in $\Gamma_1(p)$.
\smallskip

Now assume without loss of generality that $k_2\ge k_1$. Our hypotheses imply that we can then write:
$$
k_2 = k_1 + t \cdot p^{\lceil \frac{m}{e} \rceil - 1} \cdot (p-1)/d
$$
with $t$ a non-negative integer.
\smallskip

Since $p$ is odd there is a certain Eisenstein series $E$ on $\Gamma_1(p)$ of weight
$$
\kappa := (p-1)/d
$$
and $\p'$-adically integral coefficients in the field $\Q(\mu_{p-1})$ of $(p-1)$'st roots of unity with $\p'$ a prime of $\Q(\mu_{p-1})$ over $p$, and which reduces to $1$ modulo $\p'$: $E$ is the form derived from
$$
G := L(1-\kappa,\omega^{-\kappa})/2 + \sum_{n=1}^{\infty} \left( \sum_{d\mid n} \omega^{-\kappa}(d) \cdot d^{\kappa-1} \right)
$$
by scaling so that the constant term is $1$. Here, $\omega$ is the character that becomes the Teichm\"uller character when viewed as taking values in $\Z_p^{\times}$. Cf.\ Serre, \cite{Serre}, Lemme 10, and Ribet, \cite{Ribet}, \S 2.
\smallskip

Now view $E$ as having coefficients in the compositum $M$ of $K$ and $\Q(\mu_{p-1})$. Pick a prime $\p_1$ of $M$ over $\p$ and $\p'$. Then the ramification index of $\p_1$ relative to $p$ is $e$. We deduce that
$$
E^{p^{\lceil \frac{m}{e} \rceil - 1}\cdot t} \equiv 1 \bmod{\p_1^m}
$$
and so $\tilde{f} := f_1\cdot E^{p^{\lceil \frac{m}{e} \rceil - 1}\cdot t} \equiv f_1 \bmod{\p_1^m}$. As now $\tilde{f}$ is a form on $\Gamma_1(N)$ (as $N$ is divisible by $p$ and $E$ is on $\Gamma_1(p)$) of weight
$$
k_1 + p^{\lceil \frac{m}{e} \rceil - 1}\cdot t \cdot \kappa = k_2
$$
we can finish the argument in the same way as in the proof of part (ii) of Theorem \ref{theorem1}.
\end{proof}

\section{Examples} We used the mathematics software program MAGMA \cite{magma} to find examples illustrating Theorem \ref{theorem1}. We looked for examples of higher congruences and where $p$ is ramified in the field of coefficients. In the notation of this paper, what we are looking for are situations where $e > 1$ and $s \geq 1$. Here are $2$ such examples.
\smallskip

We start with
$$
f_1 = q-8q^4+20q^7+\cdots,
$$
the (normalized) cusp form on $\Gamma_0(9)$ of weight $4$ with integral coefficients, and look for congruences of the coefficients of $f_1$ and $f_2$ modulo powers of a prime above $5$, for a form $f_2$ of weight $k_2$ satisfying $k_2 \equiv 4 \pod {5 \cdot (5-1)}$.

The smallest possible choice of weight for $f_2$ is $k_2 = 24$. There is a newform $f_2$ on $\Gamma_0(9)$ of weight $24$ with coefficients in the number field $K = \Q(\alpha)$ with $\alpha$ a root of $x^4-29258x^2+97377280$. The prime $5$ is ramified in $K$ and has the decomposition $5\OO_K = \p^2\p_2$.

We have $k=24$, $N=9$, $N' = 675$ and $\mu' = 1080$, and we find that $a_{\ell}(f_1) \equiv a_{\ell}(f_2) \pod {\p^3}$ for primes $\ell \leq k\mu' /12 = 2160$ with $\ell \neq 3,5$.

Since $[K:\Q] = 4$, the Galois closure $L$ of $K$ satisfies $[L:\Q] \mid 24$ (in fact $[L:\Q] = 8$ in this case). This shows that $5 \nmid e(L,5)$, i.e., $r=0$. Since we also have $m=3$ and $e = e(\p/5) = 2$, we get $s=1$ as desired. By Theorem \ref{theorem1} we conclude that $a_{\ell}(f_1) \equiv a_{\ell}(f_2) \pod {\p^3}$ for all primes $\ell \neq 3,5$.
\smallskip

Similarly we find a newform $f_3$ on $\Gamma_0(9)$ of weight $k_3 = 44$ with coefficients in a number field $K' = \Q(\beta)$ with $\beta$ a root of
$$
x^8-438896x^6+60873718294x^4-2968020622607040x^2+40426030666768772025 ~.
$$

As before $5$ is ramified in $K'$ and has the decomposition $5\OO_{K'} = \p^4\p_2^2\p_3^2$, and thus $e=4$. One finds that $a_{\ell}(f_1) \equiv a_{\ell}(f_3) \pod {\p^5}$ for primes $\ell \leq k_2\mu' /12 = 3960$ with $\ell \neq 3,5$. The Galois closure $L'$ of $K'$ satisfies $[L':\Q] = 384 \not\equiv 0 \pod 5$, which again implies $r=0$. With $m=5$ we have $s=1$ and conclude by Theorem \ref{theorem1} that $a_{\ell}(f_1) \equiv a_{\ell}(f_3) \pod {\p^5}$ for all primes $\ell \neq 3,5$.
\medskip

We are developing a larger database of similar examples. This will be reported on elsewhere.


\end{document}